\documentclass[12pt]{amsart}%
\usepackage{amsfonts}
\usepackage{epsfig}
\usepackage{graphicx}
\usepackage{amsmath}
\usepackage{amsfonts}
\usepackage{amssymb}
\usepackage{latexsym}
\usepackage{rotating}
\usepackage{pstricks, pst-node, pst-text, pst-3d}
\usepackage{amsbsy}

\input{amssym.def}
\newtheorem{theorem}{Theorem}

\newtheorem{corollary}{Corollary}
\newtheorem{lemma}{Lemma}

\theoremstyle{remark}

\newcommand{\re}{\text{\rm Re }}
\newcommand{\im}{\text{\rm Im }}
\newcommand{\s}{\vspace{0.3cm}}
\input epsf
\begin{document}
\title[Singular solutions...]{Singular  solutions to the
L\"owner equation}
\author[D.~Prokhorov  and A.~Vasil'ev]{Dmitri Prokhorov  and Alexander Vasil'ev}

\thanks{Partially supported by
RFBR (Russia) 07-01-00120, by the grant of the Norwegian Research Council \#177355/V30, and by the
European Science Foundation RNP HCAA} \subjclass[2000]{Primary 30C35, 30C20; Secondary
30C62} \keywords{Univalent function, L\"owner equation, slit map}
\address{D.~Prokhorov: Department of Mathematics and Mechanics, Saratov State University,
Saratov 410012,
Russia}
\email{ProkhorovDV@info.sgu.ru}
\address{A.~Vasil'ev: Department of Mathematics, University of Bergen, Johannes Brunsgate 12,
Bergen 5008, Norway}
\email{alexander.vasiliev@math.uib.no}

\begin{abstract}
We consider the L\"owner differential equation in ordinary derivatives generating univalent self-maps of the unit disk (or
of the upper half-plane). If the solution to this equation represents a one-slit map, then the driving term 
is a continuous function. The reverse statement is not true in general as a famous Kufarev's example shows.
We address the following problem: to find a condition for  the L\"owner equation to generate one-slit solutions.
New examples of non-slit solutions to the L\"owner equation are presented and a comparison with the L\"owner PDE is given. Properties of singular slit solutions in the half-plane are revealed.
\end{abstract}
\maketitle

\section{Introduction}

The L\"owner parametric method has proved to be one of the powerful tools in geometric function theory by
means of which the most intriguing Bieberbach problem was finally solved  by de Branges in 1984.  
The famous L\"owner equation was introduced in a seminal 1923 paper \cite{Loewner}. Since then
many deep results were obtained most of which were related to extremal problems in the classes
of univalent functions. Stochastic version of the L\"owner equation was introduced by Schramm 
and it became an actively developing topic recently. However during the last decade, it turned out that the geometry of solutions to the classical L\"owner equation is still less known. In particular, L\"owner himself \cite{Loewner}
studied one-slit self-maps of the unit disk looking for a representation of a dense subclass of the class
of all univalent normalized functions in the unit disk. The one-slit evolution led him to the L\"owner
equation with a continuous driving term. Later in 1947, Kufarev gave an example of a solution to the L\"owner
equation with a continuous driving term, and such that the image of the unit disk under this solution
represents a family of hyperbolic half-planes. This brilliant piece was obtained in a way of explicit integration
of the L\"owner equation in some particular case. Since then, it has seemed to be a unique exception in the general
picture of slit solutions. 

Let us consider the following problem: {\it Under which conditions to the L\"owner equation with a continuous driving term the solution represents
a one-slit  map?}

The first simple sufficient condition to the L\"owner equation to have a one-slit solution can be found in
\cite[page 59]{Aleksandrov}. Namely, if the driving term has bounded first time derivative, then the solution
maps the unit disk onto itself minus a slit along a $C^1$ Jordan curve. A non-trivial sufficient condition appeared only in 2005 by Marshall and Rohde  \cite{Marshall}. The condition was given in terms of analytic properties Lip(1/2) (H\"older continuous with  exponent 1/2) of the driving term, and the 1/2-norm of it was required to be bounded. The sharp
bound 4 for this norm was found by Lind   \cite{Lind} in the same year. Observe that there is no upper bound for the driving
term which was shown in \cite{Kadanoff}.  In fact, Marshall and Rohde  \cite{Marshall} showed
that under these conditions the slit will be even quasisymmetric and situated in a Stolz angle (quasislit).

Let us observe that the driving term in Kufarev's example is also  Lip(1/2) and the 1/2-norm is equal to $3\sqrt{2}= 4.24264\dots$, which is not too far from the sharp constant 4. So it is less probable to expect the complete
answer to the problem in terms of the analytic properties of the driving term.

Our main idea is to compare the one-slit dynamics in the unit disk generated by the L\"owner ODE with  that of the subordination chains and in the PDE version of the L\"owner equation for which the L\"owner ODE is a characteristic equation. The result states that a  possibility
for the L\"owner ODE to have non-slit solutions with a continuous driving term corresponds to the slit
subordination evolution which is singular at the some moment, i.e., some non-zero area is added after this moment.
We analyze  Kufarev's example from this viewpoint and give new examples of Kufarev type. We point out that the situation with
L\"owner PDE is different. To this end we analyze a result by Pommerenke \cite{Pomme} in this direction.  Finally, we study some properties of singular solutions to the L\"owner equation at the initial moment.

\section{L\"owner equations}

In this section we give a short overview of the alternatives of the L\"owner equation we are working with.
Let us start with the classical L\"owner subordination and the corresponding L\"owner PDE. For the details
we refer to the classical Pommerenke's monograph \cite{Pommerenke}.

A L\"owner subordination chain $\Omega(t)\subset \mathbb C$ is described by the time-dependent family of conformal maps $F(z,t)$
from the unit disk $\mathbb D=\{z:\,|z|<1\}$ onto $\Omega(t)$, normalized by $F(z,t)=e^tz+a_2(t)z^2+\dots$. In the 1923 seminal L\"owner's paper \cite{Loewner}, the domain $\Omega(t)$ was the complex plane $\mathbb C$ minus
a slit along a Jordan curve with a unique finite tip going to infinity for every moment $t\in [0,\infty)$.

Given a subordination
chain of one-slit domains $\Omega(t)$ defined for $t\in [0,\infty)$, there exists
a continuous real-valued function $u(t)$,
 such that 
\begin{equation}
\dot{F}(z,t)=zF'(z,t)\frac{e^{iu(t)}+z}{e^{iu(t)}-z},\label{Loewner1}
\end{equation}
for $z\in \mathbb D$ and for  all $t\in [0,\infty)$. Here $\dot{F}$ and $F'$ stand for $t$- and $z$- derivatives respectively.
 
The initial condition
$F(\zeta,0)=F_0(\zeta)$ is not given on the characteristics of the
partial differential equation (\ref{Loewner1}), hence the solution exists
and is unique. Assuming $s$ as a parameter along the characteristics
we have $$ \frac{dt}{ds}=1,\quad \frac{dz}{ds}=-z
\frac{e^{iu(t)}+z}{e^{iu(t)}-z}, \quad \frac{dF}{ds}=0,$$ with the initial conditions
$t(0)=0$, $z(0)=\zeta$, $F(z,0)=F_0(z)$, where $z$ is in
$\mathbb D$.  Obviously, $t=s$. Observe that the domain of $z$ is the entire unit disk, however the solutions to
the second equation of the characteristic system range within the unit disk but do not fill it. 
Therefore, introducing another letter $w$ in order to distinguish the function $w(\zeta,t)$  from the variable $z$, we arrive at the Cauchy problem for the  L\"owner
equation in ordinary derivatives for a function $z=w(\zeta,t)$
\begin{equation}
\frac{dw}{dt}=-w\frac{e^{iu(t)}+w}{e^{iu(t)}-w},\label{Loewner2}
\end{equation}
with the initial condition $w(\zeta,0)=\zeta$. The equation (\ref{Loewner2}) is the non-trivial  characteristic
equation for (\ref{Loewner1}). 

In order to guarantee  the solution $F_0(w^{-1}(z,t))$ to \eqref{Loewner1} to be univalent for all $t\in [0,\infty)$, we must extend it to the whole unit disk $\mathbb D$. As it was observed in \cite{ProkhVas}, it can be done when the initial map $F_0$ is chosen to be the limit
\[
F_0(z)=\lim\limits_{t\to\infty}e^tf(z,t),\quad z\in \mathbb D,
\]
where $f(z,t)=e^{-t}z(1+c_1(t)z+\dots)$ is a solution to the equation
\begin{equation}
\frac{df}{dt}=-f\frac{e^{iu(t)}+f}{e^{iu(t)}-f}, \quad f(z,0)\equiv z, \label{Loewner3}
\end{equation}
with the same continuous driving term $u(t)$ on $t\in [0,\infty)$ as in \eqref{Loewner1}. Moreover, $f(z,t)$ can be
represented by the solution to \eqref{Loewner1} as $f(z,t)=F^{-1}(F_0(z),t)$.

Let us give here the half-plane version of the L\"owner equation. First of all, let us observe that
if $f$ is a slit solution to the equation \eqref{Loewner3}, then the endpoint of the slit on $\mathbb T=\partial \mathbb D$ may change in time $t$ as well as its shape. It makes it difficult to follow the dynamics of the slit growth as well
as its geometric properties. So the new trends in research in L\"owner theory  suggest to work with
mappings from the evolution domain to a canonical domain, the half-plane in our case,

 Let $\mathbb H=\{z: \im
z>0\}$, $\mathbb R=\partial \mathbb H$. Let us consider the growing slit $\gamma_t$ along a Jordan curve $\{w\in \gamma_t\Leftrightarrow w=\gamma(t), t\in [0,\infty)\}$
in $\mathbb H$ from the origin $\gamma(0)=0$ to a finite point of $\mathbb H$. The functions $h(z,t)$, with the hydrodynamic normalization near infinity as
$h(z,t)=z+2t/z+O(1/z^2)$,  solving the equation
\begin{equation}
\frac{dh}{dt}=\frac{2}{h-\lambda(t)}, \quad h(z,0)\equiv z, \label{Loewner4}
\end{equation}
map $\mathbb H\setminus \gamma(t)$ onto  $\mathbb H$, where $\lambda(t)$ is a real-valued
continuous driving term. 

\section{Kufarev's example and singular L\"owner maps}

As it was mentioned in Introduction, there are two known sufficient conditions that guarantee slit
solutions to the L\"owner equation \eqref{Loewner3}. The first one is found in
\cite[page 59]{Aleksandrov}. It states that  if the driving term $u(t)$ has bounded first  derivative, then the solution $f(z,t)$
maps the unit disk onto itself minus a slit along a $C^1$ Jordan curve. 

The second one belongs to Marshall and Rohde \cite{Marshall}. Their result
states that if $u(t)$ is Lip(1/2) (H\"older continuous with  exponent 1/2), and if for a certain
constant $C_{\mathbb D}>0$, the norm $\|u\|_{1/2}$ is bounded $\|u\|_{1/2}<C_{\mathbb D}$, then the solution $f(z,t)$ is a
slit map, and moreover, the Jordan arc $\gamma(t)$ is a quasislit (a quasisymmetric image of an interval within a Stolz angle). As they also proved, a
converse statement without the norm restriction holds. The absence of the norm restriction in
the latter result is essential. On one hand, Kufarev's example \cite{Kufarev} contains
$\|u\|_{1/2}=3\sqrt{2}$, which means that $C_{\mathbb D}\leq 3\sqrt{2}$. On the other hand,
Kager, Nienhuis, and Kadanoff  \cite{Kadanoff} constructed exact slit solutions to the
half-plane version of the L\"owner equation with arbitrary norms of the driving term.

The question about the slit maps and the behaviour of the driving
term $\lambda(t)$ in the case of the half-plane $\mathbb H$ was addressed by Lind \cite{Lind}.
The techniques used by Marshall and Rohde carry over to prove a similar result in the case of
the equation (\ref{Loewner4}), see \cite[page 765]{Marshall}. Let us denote by $C_{\mathbb H}$ the
corresponding bound for the norm $\|\lambda\|_{1/2}$. The main result by Lind is the sharp
bound, namely $C_{\mathbb H}=4$. As it was remarked in \cite{ProkhVas2}, $C_{\mathbb H}=C_{\mathbb D}=4$.

Let us consider Kufarev's example \cite{Kufarev} in details. Set the function $$u(t)=3\arcsin \sqrt{1-e^{-2t}}.$$
It increases from $0$ to $3\pi/2$ as $t$ varies in $[0,\infty)$. Solving equation \eqref{Loewner3} with this
driving term we obtain
\[
f(z,t)=\frac{1}{\cos u(t)}\left(z+e^{2iu(t)}-\sqrt{(1-z)(e^{2iu(t)}-z)}\right).
\]
This solution maps the unit disk onto the hyperbolic half-plane $\mathbb H_h(t)$ in the unit disk bounded by the circular arc
orthogonal to $\mathbb T$ joining the points $e^{iu(t)}$ and $e^{3iu(t)}$, $f(1)=e^{iu(t)}$, $f(e^{4iu(t)})=e^{3iu(t)}$.

Comparing Kufarev's example with the Marshall and Rohde result we see that the above driving term is Lip(1/2) and the 1/2-norm is equal to $3\sqrt{2}= 4.24264\dots$, i.e., it is very close to the Marshall and Rohde condition.  Therefore, the engine
forcing the equation to generate such a singular behavior differs from just analytic properties of the driving term. 

Let us consider the corresponding subordination evolution and the solution $F$ to the equation \eqref{Loewner1}.
The map $F_0$ is given as
\[
F_0(z)=\lim\limits_{t\to \infty}e^tf(z,t)=\frac{z}{1-z}.
\]
It maps the unit disk onto the half-plane $\re w>-\frac{1}{2}$. The solution $F$ to the equation \eqref{Loewner1} is given
by the formula $F(w,t)=F_0(f^{-1}(w,t))$, which in the explicit form becomes
\[
F(w,t)=\frac{e^t w-e^{-2i\alpha(t)}w^2}{(1-e^{-i\alpha(t)}w)^2},\quad \alpha=\arccos (e^{-t})\in[0,\pi/2).
\]
The function $F(w,t)$ maps the hyperbolic half-plane $\mathbb H_h(t)$ onto the half-plane $\re w>-\frac{1}{2}$. By reflection we extend $F(w,t)$ into the whole disk $\mathbb D$ and the extension $F(z,t)$ maps the unit disk onto the complex plane $\mathbb C$ minus the slit along the vertical ray
 $$\{w: w=-\frac{1}{2}+iy, y\in(-\infty, \frac{1}{2}\cot 2\alpha(t)]\}.$$
Now it becomes clear that the singular behavior of Kufarev's example is due to the topology change in the image of $\mathbb D$ by $F(z,t)$ after the initial moment $t=0$. In fact, we add a non-zero area at the initial moment.

Based on above considerations let us give some answer to the problem formulated in Introduction. Let $F_0(z)=z+a_2z^2+\dots$ be a conformal map of the disk $\mathbb D$ onto the domain $\Omega_0\subset \mathbb C$, $0\in \Omega_0$, bounded by a curve $\Gamma=\{\Gamma(t),t\in(0,\infty)\}$, which is a homeomorphic image of the open interval $(0,\infty)$, and such that its closure $\hat\Gamma$ is $\partial \Omega_0$. By construction, it is clear that $\hat{\Gamma}$ meets itself at most once (possibly at infinity). If it meets itself, then  the complement to $\Omega_0\cup \hat{\Gamma}$ has non-zero (possibly infinite) area. Without loss of generality let us assume
that the right endpoint $\infty$ of the interval corresponds to $\infty\in \hat{\Gamma}$. There exists a point $t_0\in(0,\infty]$
such that $\Gamma(t_0)= \lim_{t\to +0}\Gamma(t)$. Denote by $\Gamma_t=\Gamma[t,\infty)$. We choose the parametrization of $\Gamma_t$, $t\in (0,\infty)$ such that the conformal radius of $\mathbb C \setminus \Gamma_t$ is equal to $e^t$. Now let us construct the subordination chain of functions $F(z,t)$ that map $\mathbb D$ onto $\mathbb C \setminus \Gamma_t$. It satisfies the L\"owner equation \eqref{Loewner1} with some continuous driving term $u(t)$.
We construct $f(z,t)=F^{-1}(F_0(z),t)$, $z\in \mathbb D$. It satisfies the L\"owner ODE \eqref{Loewner3} with the same
driving term. At the same time the complement of $f(\mathbb D,t)$ to $\mathbb D$ has non-zero area, and therefore,
$f(z,t)$ is not a slit map.

If the curve $\hat{\Gamma}$ does not meet itself, then $f(z,t)=F^{-1}(F_0(z),t)$ represents a slit evolution.

We formulate above in the following theorem.

\begin{theorem}\label{th1}
Slit evolution in the unit disk given by the solution $f(z,t)$ to the L\"owner ODE \eqref{Loewner3} is controlled
by the subordination evolution given by the solution $F(z,t)$ to the corresponding L\"owner PDE \eqref{Loewner1}. More precisely,
if $F(\mathbb D,0)$ is bounded by the above defined curve $\hat{\Gamma}$ that meets itself once, then $f$ does not represent a slit
evolution at any time, moreover this evolution is of Kufarev type: the complement of $f(\mathbb D,t)$ to $\mathbb D$ has non-zero area. If  $\hat{\Gamma}$ is a Jordan curve, then $f$ represents a slit evolution.
\end{theorem}

To clear up the complete picture let us analyze the analogous problem with subordination chains and with the L\"owner PDE \eqref{Loewner1}.
Pommerenke \cite{Pomme} gave a necessary and sufficient condition for the geometry of a subordination chain of domains so that  the 
corresponding subordination chain of mapping functions satisfy the L\"owner PDE \eqref{Loewner1} with a continuous driving term. Namely, he proved the following result \cite[Theorem 1]{Pomme} (slightly reformulated for our setup and notations).  

\begin{theorem}\label{thPomme} {\rm (Pommerenke \cite{Pomme}).}
Let $F(z,t)$ be a subordination chain
of functions normalized as $F(z,t)=e^tz+a_2(t)z^2+\dots$ in the unit disk $\mathbb D$ corresponding to the subordination chain of domains $\Omega(t)$, $F(\mathbb D,t)=\Omega(t)$, $t\geq 0$. The functions $F(z,t)$ satisfy the L\"owner PDE \eqref{Loewner1} with a continuous driving term
$u(t)$, if and only if, for every $\varepsilon>0$ there exists  $\delta>0$, such that whenever $0\leq t-s\leq \delta$, $s,t\geq 0$, some cross-cut C of $\Omega(t)$ with the spherical diameter $<\varepsilon$ separates 0 from $\Omega(t)\setminus \Omega(s)$.
\end{theorem}

Theorems 1 and 2 show that the subordination chain can be based on slit erasing whereas the dynamics in the unit disk
is non-slit. 

\s
\noindent
{\sc Example.} Let us give an example of the function $f(z,t)=e^{-t}z+c_1z^2+\dots$, that satisfies the L\"owner equation \eqref{Loewner3},
and for each fixed $t$ maps the unit disk $\mathbb D$ onto $\mathbb D$ minus a region with non-zero area.  Let $F_0(z)\equiv z$. The map $f(z,t)=F^{-1}(z,t)$ possesses the required properties, where
\[
F(z,t)=1-\frac{1}{w\circ\zeta(e^{-i\alpha}z,t)}.
\]
Here the function
\[
\zeta=i\frac{1-a+z(1-\bar{a})}{1+a-z(1+\bar{a})},\quad \mbox{with \ } a=a(\lambda)=\frac{1+i\zeta_0(\lambda)}{1-i\zeta_0(\lambda)},
\]
maps the unit disk onto the half-plane $\{\zeta: \, \im \zeta >0\}$, the origin is mapped onto the point $\zeta_0(\lambda)$ which is a unique solution
with the positive imaginary part to the equation
\[
(3\zeta^2+3(1+\lambda)\zeta-\lambda)\zeta^{-3/2}=-2(3+\lambda),\quad \lambda=\lambda(t)\in(-3,0).
\]
The function 
\[
w=w(\zeta,t)=\frac{1}{2}\left(1-\frac{3\zeta^2+3(1+\lambda)\zeta-\lambda}{2(3+\lambda)}\zeta^{-3/2}\right)
\]
maps the upper half-plane $\{\zeta: \, \im \zeta >0\}$ onto $\mathbb C$ minus two rectilinear slits. The first one is along the negative
real axis $(-\infty, 0]$, and the second is along the vertical ray \[\{w:\, w=\frac{1}{2}+iy,\, y\in (-\infty,-\frac{1+3\lambda}{2(3+\lambda)}(-\lambda)^{-1/2}]\}.\]
The points 1 and $\lambda$ are mapped onto the  finite tips of these rays respectively, and the point $\zeta_0$ is mapped onto 1. This map can be found, e.g., in \cite{Koppenfels}. The function $1-\frac{1}{w}$ maps the above configuration onto $\mathbb C$ minus a slit consisting of two parts $\gamma_1$ and $\gamma_2$, where $\gamma_1=[1,\infty)$ and $\gamma_2$ is the circular arc 
\[
\gamma_2=\{z:\, z=e^{i\theta},\, \theta\in (2\arctan \frac{3+\lambda}{1+3\lambda}\sqrt{-\lambda},2\pi].\}
\]
The parameter $\lambda=\lambda(t)$ is defined from the equation
\[
\frac{3}{8(3+\lambda)}|\zeta_0^{-1/2}(\lambda)-(1+\lambda)\zeta_0^{-3/2}(\lambda)+\lambda \zeta_0^{-5/2}(\lambda)|\frac{1-|a(\lambda)|^2}{|1+a^2(\lambda)|}=2e^t,
\]
and the parameter $\alpha$ is
\[
\alpha(t)=\frac{3\pi}{2}-2\arg(1+a)+\arg(\zeta_0^{-1/2}(\lambda)-(1+\lambda)\zeta_0^{-3/2}(\lambda)+\lambda \zeta_0^{-5/2}(\lambda)).
\]
The function $F^{-1}(z,t)$ maps $\mathbb C\setminus (\gamma_1\cup\gamma_2)$ onto the whole unit disk $\mathbb D$. The function $F_0(z)$ is the identity map in $\mathbb D$. Therefore, $f(z,t)=F^{-1}(F_0(z),t)\equiv F^{-1}(z,t)$ maps $\mathbb D$ onto $\mathbb D$ minus a region bounded by the arc
$F^{-1}(\gamma_2,t)$ and the arc of $\mathbb T$ defined by the endpoints of $\gamma_2$. The slit evolution of the function $F$
assures that it satisfies the L\"owner equation \eqref{Loewner1} with some continuous driving term $u(t)$, and therefore, the function
$f$ satisfies the L\"owner equation \eqref{Loewner3} with the same driving term. 

\begin{figure}[ht] \scalebox{0.7}{
\begin{pspicture}(3,1)(17,7)
\pscircle[linewidth=0.15mm, fillstyle=solid,
fillcolor=lightgray](5,4){2} \rput(4.7,3.7){0}
\rput(5.3,7){$\eta$}\rput(8,4.3){$\xi$}
\rput(3.7,3){$\mathbb D$}\rput(7.3,3.7){$1$}
\psline[linewidth=0.15mm]{->}(5,1)(5,7)
\psline[linewidth=0.15mm]{->}(1,4)(8,4)
 \pscircle[fillstyle=solid,
fillcolor=black](5,4){.1}
\pscircle[linecolor=blue,linewidth=1mm, fillstyle=solid,
fillcolor=lightgray](15,4){2} \rput(14.7,3.7){0}
\rput(15.3,7){$y$}\rput(18,4.3){$x$}
\rput(13.7,3){$U$}\rput(17.3,3.7){$1$}
\psline[linewidth=0.15mm]{->}(15,1)(15,7)
\psline[linewidth=0.15mm]{->}(11,4)(18,4)
 \pscircle[fillstyle=solid,
fillcolor=black](15,4){.1}
\pscurve[linewidth=0.8mm,
linecolor=red]{->}(8,5)(9.5,5.5)(11,5)
\rput(9.5,5.9){$F_0(\zeta)\equiv\zeta$}
\end{pspicture}}
\end{figure}

\begin{figure}[ht] \scalebox{0.7}{
\begin{pspicture}(3,1)(17,7)
\pscircle[linewidth=0.15mm, fillstyle=solid,
fillcolor=lightgray](5,4){2} \rput(4.7,3.7){0}
\rput(5.3,7){$\eta$}\rput(8,4.3){$\xi$}
\rput(3.7,3){$\mathbb D$}\rput(7.3,3.7){$1$}
\psline[linewidth=0.15mm]{->}(5,1)(5,7)
\psline[linewidth=0.15mm]{->}(1,4)(8,4)
 \pscircle[fillstyle=solid,
fillcolor=black](5,4){.1}
\psframe[linecolor=lightgray, fillstyle=solid,
fillcolor=lightgray](10.5,0.5)(18.5,7.5)
\psarc[linecolor=blue,linewidth=1mm](15,4){2}{0}{320} \rput(14.7,3.7){0} 
\rput(15.3,7){$y$}\rput(18,4.3){$x$}
\rput(13.7,3){$U$}\rput(17.3,3.7){$1$}
\psline[linewidth=0.15mm]{->}(15,1)(15,7)
\psline[linewidth=0.15mm]{->}(11,4)(18,4)
\psline[linecolor=blue,linewidth=1mm](17,4)(18.5,4)
 \pscircle[fillstyle=solid,
fillcolor=black](15,4){.1}
\pscurve[linewidth=0.8mm,
linecolor=red]{->}(8,5)(9.5,5.5)(11,5)
\rput(9.5,5.9){$F(\zeta,t)$}
\end{pspicture}}
\caption[]{Subordination chain in Example}\label{fig4}
\end{figure}
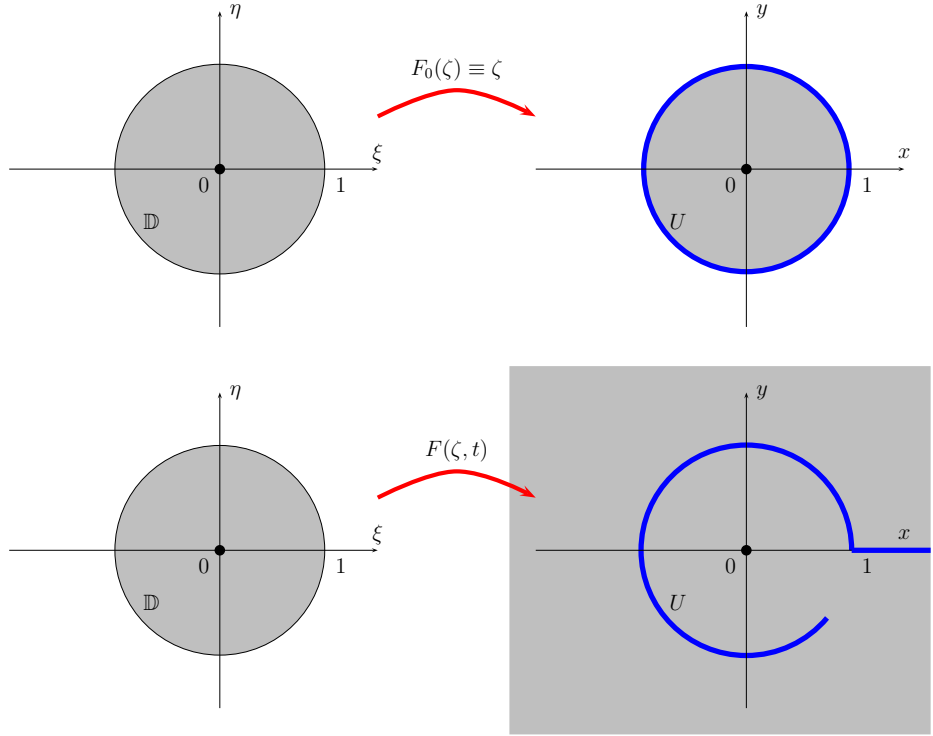

Another possible scenario satisfying the conditions of Theorem~2 is shown in Figure~2.

\begin{figure}[ht] \scalebox{0.7}{
\begin{pspicture}(3,1)(17,7)
\rput(1.7,1.7){0}
\rput(2.3,7){$\eta$}\rput(8,2.3){$\xi$}
\psline[linewidth=0.15mm]{->}(2,1)(2,7)
\psline[linewidth=0.15mm]{->}(1,2)(8,2)
\pscurve[linewidth=1mm,linecolor=blue](3, 7.5)(3.1,5)(3.2,4.2)(3.5,3)(3.7,2.6)(4,2.5)(5,3)(5,3.5)(3.1,4.6)
\pscurve[linewidth=1mm,linecolor=blue](3.1, 4.6)(3.5,5.1)(4,5.6)
 \pscircle[fillstyle=solid,
fillcolor=black](2,2){.1}
\rput(5.5,6.5){$F_0(z)$}
\psarc[linecolor=blue,linewidth=1mm](15,4){2}{60}{370} \rput(14.7,3.7){0} 
\pscurve[linewidth=1mm,linecolor=blue](16,5.75)(16.1,4.7)(17,4.3)
\psline[linewidth=1mm,linecolor=blue](16,5.75)(15.3,5.3)
\rput(15.3,7){$y$}\rput(18,4.3){$x$}
\rput(13.7,3){$U$}\rput(17.3,3.7){$1$}
\psline[linewidth=0.15mm]{->}(15,1)(15,7)
\psline[linewidth=0.15mm]{->}(11,4)(18,4)
 \pscircle[fillstyle=solid,
fillcolor=black](15,4){.1}
\rput(16.5,6.5){$f(z,t)$}
\rput(0.5,6.5){{\Large (a)}}
\rput(10.5,6.5){{\Large (b)}}
\end{pspicture}}
\caption[]{(a) $F(z,t)$ is a result of slit erasing for $F_0$;  \newline (b) The dynamics of $f=F^{-1}\circ F_0.$}
\end{figure}
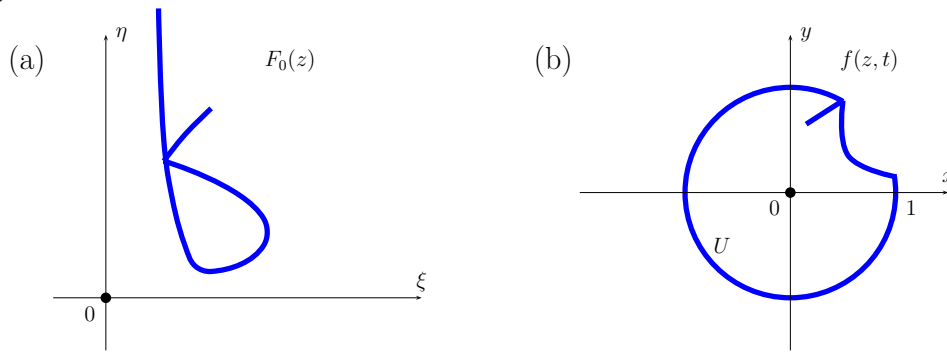

\section{Slit maps in the half-plane}

The half-plane version of the L\"owner equation deals with $\mathbb H=\{z: \im z>0\}$,
$\mathbb R=\partial \mathbb H$, and the functions $h(z,t)$, which solving  equation \eqref{Loewner4} are normalized near infinity by
$h(z,t)=z+2t/z+O(1/z^2)$.

Solutions $f(z,t)$ to equation \eqref{Loewner3} and $h(z,t)$ to equation \eqref{Loewner4}  differ in their normalization.
The coefficient $e^{-t}$ in the expansion of $f(z,t)$ is the conformal radius of $\mathbb
D\setminus \gamma(t)$, where $\gamma(t)$ is a slit along a Jordan curve starting at a point of $\mathbb T$ and ending at an interior non-zero point of $\mathbb D$, $0\not\in \gamma$. Earle and Epstein \cite{Earle} proved that if $\gamma$ has a real
analytic parametric representation $\gamma(s)$ in $(0,S]$, $\gamma(0)=1$, then the
conformal radius of $\mathbb D\setminus\gamma([s,S])$, $0<s<S$, at the origin is a real
analytic function of $s$ in $(0,S]$. In particular, $\gamma(s)$ can be the arc-length
parametrization. Hence, $t=t(s)$ and $s=s(t)$ are real analytic functions in $(0,S]$ and
$(0,T]$ respectively. Earle and Epstein \cite{Earle} also showed that  the driving term $u$ in \eqref{Loewner3} was at least $C^{n-1}$ for $C^n$-smooth slits. For $n=2$, this was extended to the situation where the parametrization $\gamma(s)$ was slightly less than $C^2$. Namely, the driving function $u$ is $C^1$ if $\gamma(s)$ is $C^1$ in $[0,S]$, $\gamma(s)$ is twice differentiable in a set $E\subset[0,S]$ of full measure and its second derivative  is locally bounded and continuous in $E$. 

The function $h(z,t)$ in \eqref{Loewner4} has the hydrodynamic normalization near infinity.
Therefore, the coefficient $2t$ at $z^{-1}$ is similar to the conformal radius $e^{-t}$ in the
disk version. The results of Earle and Epstein can be applied to the half-plane version so that
$t=t(s)$ and $s=s(t)$ are real analytic functions on $(0,S]$ and $(0,T]$ respectively for the slit 
$\gamma$ in $\mathbb H$. In the following sections we will focus on the half-plane version \eqref{Loewner4}.

The question we are considering here is concerned with the behavior of $s(t)$ at $t=0$. To pose the problem
assume that $\gamma(s)=x(s)+iy(s)$ is analytic on $[0,S]$ where $x(s)$ is even and $y(s)$ is
odd. This implies that $\gamma(t)\cup\overline{\gamma}(t)\cup\gamma(0)$ is an analytic curve
symmetric with respect to $\mathbb R$. Here we denote by $\overline{\gamma}$ the reflection of $\gamma$ with respect to
the real axis. Suppose that the L\"owner equation \eqref{Loewner4} with the
driving term $\lambda(t)$ generates a map $h(z,t)$ from $\Omega(t)=\mathbb
H\setminus\gamma(t)$ onto $\mathbb H$. Extend $h$ to the boundary $\partial\Omega(t)$ and
obtain a correspondence between $\gamma(t)\subset\partial\Omega(t)$ and a segment
$I(t)\subset\mathbb R$ while the remaining boundary part $\mathbb
R=\partial\Omega(t)\setminus\gamma(t)$ corresponds to $\mathbb R\setminus I(t)$. The image
$I(t)$ of $\gamma(t)$ can be described by solutions $h(\gamma(0),t)$ to \eqref{Loewner4} but the
initial data $h(\gamma(0),0)=\gamma(0)$ forces $h$ to be singular at $t=0$. There are two
singular solutions $h^-(\gamma(0),t)$ and $h^+(\gamma(0),t)$ such that
$I(t)=[h^-(\gamma(0),t),h^+(\gamma(0),t)]$.

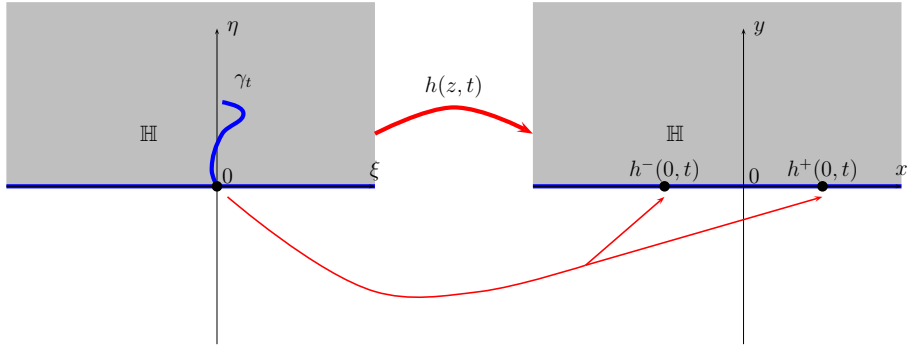
\begin{figure}[ht] \scalebox{0.7}{
\begin{pspicture}(2,1)(17,8)
\psframe[linecolor=lightgray, fillstyle=solid,
fillcolor=lightgray](11,4)(18,7.5)
\psline[linecolor=blue,linewidth=0.8mm](11,4)(18,4)
\rput(15.3,7){$y$}\rput(18,4.3){$x$}
\rput(13.7,5){$\mathbb H$}\rput(15.2,4.2){0}
\psline[linewidth=0.15mm]{->}(15,1)(15,7)
\psline[linewidth=0.15mm]{->}(11,4)(18,4)
\psframe[linecolor=lightgray, fillstyle=solid,
fillcolor=lightgray](1,4)(8,7.5)
\psline[linecolor=blue,linewidth=0.8mm](1,4)(8,4)
\rput(5.3,7){$\eta$}\rput(8,4.3){$\xi$}
\rput(3.7,5){$\mathbb H$}\rput(5.2,4.2){0}
\psline[linewidth=0.15mm]{->}(5,1)(5,7)
\psline[linewidth=0.15mm]{->}(1,4)(8,4)
\pscurve[linecolor=blue,linewidth=0.8mm](5,4)(4.9,4.3)(5.1,5)(5.5,5.4)(5.1,5.6)
\rput(5.5,6){$\gamma_t$}
 \pscurve[linewidth=0.8mm,
linecolor=red]{->}(8,5)(9.5,5.5)(11,5)
\rput(9.5,5.9){$h(z,t)$}
\pscurve[linewidth=0.3mm,
linecolor=red](5.2,3.8)(8,2)(10,2)(12,2.5)
\psline[linecolor=red,linewidth=0.3mm]{->}(12,2.5)(13.5,3.8)
\psline[linecolor=red,linewidth=0.3mm]{->}(12,2.5)(16.5,3.8)
\rput(13.5,4.3){$h^-(0,t)$}
\rput(16.5,4.3){$h^+(0,t)$}
 \pscircle[fillstyle=solid,
fillcolor=black](13.5,4){.1}
 \pscircle[fillstyle=solid,
fillcolor=black](16.5,4){.1}
 \pscircle[fillstyle=solid,
fillcolor=black](5,4){.1}
\end{pspicture}}
\caption[]{(a) The mapping $h(z,t)$}
\end{figure}

Without loss of generality, assume that $\gamma(0)=0$, which implies  $\lambda(0)=0$. By the
symmetry principle $h(z,t)$ can be extended conformally to the map from $\mathbb
C\setminus(\gamma(t)\cup\overline{\gamma}(t)\cup0)$ onto $\mathbb C\setminus I(t)$. Moreover
$h(z,t)$ is analytic in $\mathbb C$ except for two points $z=\gamma(t)$ and
$z=\overline{\gamma}(t)$, while its inverse $h^{-1}(w,t)$ is analytic everywhere except for
$w=h^-(0,t)$ and $w=h^+(0,t)$. In a neighborhood of one of prime ends at $z=0$ the function
$w=h(z,t)$ is expanded in the series
\begin{equation}h(z,t)=h^+(0,t)+a_2(t)z^2+a_3(t)z^3+\dots,\;\;\;t>0,\;\;\;a_2(t)\neq0.
\label{EXP1}
\end{equation} Hence, near $w=h^+(0,t)$, \begin{equation}
h^{-1}(w,t)=b_1(w-h^+(0,t))^{1/2}+b_2(w-h^+(0,t))+\dots,\;\;\; b_1(t)=a_2(t)^{-1/2}.
\label{EXP2}
\end{equation} The expansions about the second prime end at $z=0$ for $h(z,t)$ and about
$h^-(0,t)$ for $h^{-1}(w,t)$ are analogous.

\section{Coefficient growth for slit maps}

Prokhorov and Vasil'ev \cite{ProkhVas} studied singular solutions to the L\"owner
differential equation \eqref{Loewner4} for slit maps $h(z,t)$ generated by the driving term
$\lambda$. In particular, if $\lambda\in\text{Lip}(1/2)$ with $\|\lambda\|_{1/2}=c$, and
$\gamma(t)$ is a quasisymmetric curve, then $$\lim_{t\to+0}\sup\frac{h^+(0,t)}{\sqrt
t}\leq\frac{c+\sqrt{c^2+16}}{2}.$$ Developing this motivation we will show  how the
L\"owner differential equation \eqref{Loewner4} leads to coefficient estimates for singular
solutions. 
\begin{theorem}\label{th2} Let the L\"owner differential equation \eqref{Loewner4} with the driving term
$\lambda\in\text{Lip}(1/2)$, generate slit maps $h(z,t): \mathbb H\setminus\gamma(t)\to\mathbb
H$ where $\gamma(t)\cup\overline{\gamma}(t)\cup\gamma(0)$ is an analytic curve which is mapped
onto $[h^-(0,t),h^+(0,t)]$. Suppose that $$\lim_{t\to+0}\frac{\lambda(t)}{\sqrt t}=c,\;\;\;
\lim_{t\to+0}\frac{h^+(0,t)}{\sqrt t}=b,\;\;\;c<b\leq\frac{c+\sqrt{c^2+16}}{2}.$$ Then,
for $h(z,t)$ expanded by (\ref{EXP1}) and for every $\varepsilon>0$, we have
$$\lim_{t\to+0}a_2(t)t^{\frac{2}{(b-c)^2}+\varepsilon}=0, \quad \mbox{and \ } \lim_{t\to+0}a_2(t)t^{\frac{2}{(b-c)^2}-\varepsilon}=\infty.$$
\end{theorem}
\begin{proof} The inequality $$b\leq\frac{c+\sqrt{c^2+16}}{2}$$ was proved in  \cite{ProkhVas} for $c=\|\lambda\|_{1/2}$. Let us show that
this inequality remains valid for $c=\lim_{t\to +0}\frac{\lambda(t)}{\sqrt{t}}$.

 Indeed, the function
$\varphi(t):=h^+(0,t)/\sqrt t$ solves the differential equation $$t\varphi'(t)=\frac{2\sqrt
t}{\sqrt t\varphi(t)-\lambda(t)}-\frac{\varphi(t)}{2}.$$ Taking into account that
$\lambda(t)<\sqrt t\varphi(t)$, $t>0$, we note that $\varphi'(t)>0$ only when
$$\frac{\lambda(t)}{\sqrt t}<\varphi(t)<\varphi_1(t):=\frac{\lambda(t)}{2\sqrt
t}+\sqrt{\frac{\lambda^2(t)}{4t}+4}.$$ For every $\varepsilon>0$, the function $\varphi(t)$ does
not exceed $A(\varepsilon):=(c+\varepsilon+\sqrt{(c+\varepsilon)^2+16})/2$ in an interval $0<t<\delta(\varepsilon)$.
Otherwise $\varphi'(t^*)<0$ for some $t$,  $0<t^*<\delta(\varepsilon)$, and $\varphi(t)$ increases
as $t$ runs from $t^*$ to +0. This leads to the differential inequality
$$\frac{dh^+(0,t)}{dt}<\frac{2}{\sqrt t(A(\varepsilon)-c-\varepsilon)},\;\;\;0<t<t^*,$$ and after
integrating contradicts the theorem conditions.

The extended map $h(z,t)$ satisfies equation (\ref{Loewner4}) and its derivative $h'(z,t)$ with
respect to $z$ vanishes at $z=0$. So $w=h(z,t)$ is expanded in the series by (\ref{EXP1}) in a
neighborhood of $z=0$. Let us differentiate (\ref{Loewner4}) with respect to $z$ and let us obtain the following
differential equation $$\frac{dh'}{dt}=\frac{-2h'}{(h-\lambda(t))^2}.$$ Differentiating this
equation again we obtain
$$\frac{dh''}{dt}=-2\frac{h''(h-\lambda(t))-2h'^2}{(h-\lambda(t))^3}.$$ Putting $z=0$, we come to the
singular differential equation $$\frac{da_2}{dt}=\frac{-2a_2}{(h(0,t)-\lambda(t))^2},$$ which
gives that
\begin{equation}\frac{1}{a_2}\frac{da_2}{dt}=\frac{-2}{t((b-c)+o(1))^2},\;\;\;t\to+0.
\label{AS}
\end{equation} Integrating this asymptotic differential equation in $(0,\delta)$ one arrives at the
estimates $$ Bt^{-2/(b-c+\varepsilon)^2} \leq|a_2(t)|\leq Bt^{-2/(b-c-\varepsilon)^2},$$
 $0<t<\delta(\varepsilon)$, with a certain
$B=B(\varepsilon)$. This completes the proof.
\end{proof}

Theorem \ref{th2} establishes also the growth of the first coefficient for the inverse function
because of the connection between the coefficients $a_2(t)$ in (\ref{EXP1}) and $b_1(t)$ in
(\ref{EXP2}).
\begin{corollary}\label{cor2} Under the conditions of Theorem  \ref{th2}, for every $\varepsilon>0$ we have
$$\lim_{t\to+0}b_1(t)t^{\frac{-1}{(b-c)^2}+\varepsilon}=0\quad \mbox{and \ } \lim_{t\to+0}b_1(t)t^{\frac{-1}{(b-c)^2}-\varepsilon}=\infty.$$
\end{corollary}

Equation (\ref{Loewner4}) provides a chance to estimate the growth of coefficients $a_n$ in the
series (\ref{EXP1}). To this purpose we rewrite (\ref{Loewner4}) as $$\frac{dh(z,t)}{dt}=
\frac{2}{h(z,t)-\lambda(t)}= \frac{2}{h^+(0,t)-\lambda(t)}
\frac{1}{\frac{h(z,t)-h^+(0,t)}{h^+(0,t)-\lambda(t)}+1}$$ $$=\frac{2}{h^+(0,t)-\lambda(t)}
\sum_{k=0}^{\infty}(-1)^k\left(\frac{h(z,t)-h^+(0,t)}{h^+(0,t)-\lambda(t)}\right)^k$$
$$=\sum_{k=0}^{\infty}\frac{2(-1)^k}{(h^+(0,t)-\lambda(t))^{k+1}}
\left(\sum_{n=2}^{\infty}a_n(t)z^n\right)^k.$$ Equating coefficients at $z^n$ in the both sides
of this equation we obtain recurrent singular linear differential equations for $a_n(t)$.
Under the conditions of Theorem 3, we observe that there exists $\alpha<0$, such that for all
$n\geq2$, $$a_n(t)=O(t^{\alpha n}),\;\;\;t\to+0.$$

\section{Coefficient growth for the inverse function}

We have to study also the coefficient growth for the inverse function $h^{-1}(w,t)$ expanded
by (\ref{EXP2}).
\begin{theorem} Let the L\"owner differential equation (\ref{Loewner4}) with the driving term
$\lambda\in\text{Lip}(1/2)$, generate slit maps $h(z,t): \mathbb H\setminus\gamma(t)\to\mathbb
H$, where $\gamma(t)\cup\overline{\gamma}(t)\cup\gamma(0)$ is an analytic curve which is mapped
onto $[h^-(0,t),h^+(0,t)]$. Suppose that $$\lim_{t\to+0}\frac{\lambda(t)}{\sqrt t}=c,\;\;\;
\lim_{t\to+0}\frac{h^+(0,t)}{\sqrt t}=b,\;\;\;c<b\leq\frac{c+\sqrt{c^2+16}}{2}.$$ Given
$\varepsilon>0$, the coefficients $b_n(t)$ in the expansion (\ref{EXP2}) for $g^{-1}(w,t)$ and
for odd $n>1$, satisfy the inequality $$|b_n(t)|\leq
A_nt^{\frac{1}{(b-c)^2}-\frac{n-1}{4}-\varepsilon},\;\;\;0<t<\delta,$$ with $A_n$ depending only on
$n$ and with $\delta$ depending on $\varepsilon$.
\end{theorem}
\begin{proof}
The function $h^{-1}(w,t)$ solves the differential equation \begin{equation}
\frac{dh^{-1}(w,t)}{dt}=-(h^{-1}(w,t))'\frac{2}{w-\lambda(t)},
\end{equation}
where $(h^{-1}(w,t))'$ denotes the derivative of $h^{-1}(w,t)$ with respect to $w$. Expanding the
right-hand side in the series near $w=h^+(0,t)$ we obtain $$\frac{dh^{-1}(w,t)}{dt}=
\frac{-2(h^{-1}(w,t))'}{h^+(0,t)-\lambda(t)}
\frac{1}{\frac{w-h^+(0,t)}{h^+(0,t)-\lambda(t)}+1}=$$
$$\frac{-2(h^{-1}(w,t))'}{h^+(0,t)-\lambda(t)}
\sum_{k=0}^{\infty}(-1)^k\left(\frac{w-h^+(0,t)}{h^+(0,t)-\lambda(t)}\right)^k.$$ Let us substitute
here the expansion (\ref{EXP2}) which converges for $|w-h^+(0,t)|<h^{+}(0,t)-h^-(0,t)$ and diverges
for $|w-h^+(0,t)|>h^{+}(0,t)-h^-(0,t)$. We  rewrite the latter differential equation as \begin{equation}
\frac{d}{dt}\sum_{n=1}^{\infty}b_n(t)(w-h^+(0,t))^{n/2}= \label{CON} \end{equation}
$$\sum_{n=1}^{\infty}nb_n(t)(w-h^+(0,t))^{n/2-1}
\sum_{k=0}^{\infty}\frac{(-1)^{k-1}(w-h^+(0,t))^k}{(h^+(0,t)-\lambda(t))^{k+1}}.$$ Equating
coefficients at the same powers in the both sides of (\ref{CON}) one obtains recurrent singular
linear differential equations for $b_n(t)$. We start with positive powers because powers
(-1/2) and 0 produce trivial equations. The equation \begin{equation} \frac{db_n(t)}{dt}=
\sum_{j=0}^{[\frac{n-1}{2}]}\frac{(-1)^j(n-2j)b_{n-2j}}{(h^+(0,t)-\lambda(t))^{j+2}}
\label{COE}
\end{equation}
holds, where $[a]$ is the integer part of $a\geq0$. Note that, for every $n\geq1$, equation
(\ref{COE}) contains only coefficients with either even or odd indices.

Let us show that $|b_n(t)|\leq A_nt^{1/(b-c)^2-(n-1)/4-\varepsilon}$ for every $\varepsilon>0$, for odd $n>1$, 
and for $A_n$ depending on $n$. Given $\varepsilon>0$, the solution $b_n(t)$ to equation (\ref{COE})
satisfies the inequality
\begin{equation}|b_n(t)|\leq C_n' t^{\frac{n}{(b-c)^2}-n\varepsilon}
\left(\int^tt^{\frac{-n}{(b-c)^2}} \sum_{j=1}^{[\frac{n-1}{2}]}
\frac{(n-2j)|b_{n-2j}|t^{-\frac{j+2}{2}}}{(b-c-\varepsilon)^{j+2}}dt\right),\;\;\; 0<t<\delta.
\label{BN1}
\end{equation} This inequality proves the assertion of Theorem  for $n=3$. Suppose that the
assertion is true for $n=3,5,\dots,n-2$. Then, for $0<t<\delta$, \begin{equation} |b_n(t)|\leq
C_nt^{\frac{n}{(b-c)^2}-n\varepsilon}\int^tt^{\frac{-n}{(b-c)^2}}
\sum_{j=1}^{[\frac{n-1}{2}]}t^{\frac{1}{(b-c)^2}-\frac{n+3}{4}}dt\leq
A_nt^{\frac{1}{(b-c)^2}-\frac{n-1}{4}-\varepsilon}, \label{BN2} \end{equation} which proves the
induction conjecture and completes the proof.
\end{proof}

The equation (\ref{COE}) for $b_1(t)$ corresponds to the similar equation in Theorem \ref{th2} for
$a_2(t)$ and Corollary \ref{cor2}.

\begin{corollary} Under the conditions of Theorem \ref{th2}, given $\varepsilon>0$, we have
$$\lim_{t\to+0}b_n(t)t^{-\frac{1}{(b-c)^2}+\frac{n-1}{4}+\varepsilon}=0.$$
\end{corollary}

A similar statement for even $n$ is true with slightly changed powers since asymptotic
behavior of $b_2(t)$ is equal to that of $b_1^2(t)$.

\section{Singularity of the slit parametrization}

Let us examine the type of singularity  of the parametrization $\gamma=\gamma(t)$. Assume in this section that $c\geq 0$,
otherwise we apply all reasonings to $h^-(0,t)$ instead of $h^+(0,t)$.

\begin{lemma}\label{lemma1}  Let the L\"owner differential equation (\ref{Loewner4}) with the driving term
$\lambda\in\text{Lip}(1/2)$, generate slit maps $h(z,t): \mathbb H\setminus\gamma(t)\to\mathbb
H$, where $\gamma(t)\cup\overline{\gamma}(t)\cup\gamma(0)$ is an analytic curve which is mapped
onto $[h^-(0,t),h^+(0,t)]$. Suppose that $$\lim_{t\to+0}\frac{\lambda(t)}{\sqrt t}=c\geq 0,\;\;\;
\lim_{t\to+0}\frac{h^+(0,t)}{\sqrt t}=b,\;\;\;c<b\leq\frac{c+\sqrt{c^2+16}}{2}.$$ Then,
given $\varepsilon>0$, we have $$\lim_{t\to+0}\gamma(t)t^{-\frac{1}{(b-c)^2}-\frac{1}{4}+\varepsilon}=0.$$
\end{lemma}

\begin{proof} We write
$$\gamma(t)= h^{-1}(\lambda(t),t)=\sum_{n=1}^{\infty}b_n(t)(\lambda(t)-h^+(0,t))^{n/2},$$ or
\begin{equation}\gamma(t)t^{-\frac{1}{(b-c)^2}-\frac{1}{4}+\varepsilon}=
\sum_{n=1}^{\infty}b_n(t)t^{-\frac{1}{(b-c)^2}+\frac{n-1}{4}+\varepsilon}
\left(\frac{\lambda(t)-h^+(0,t)}{\sqrt t}\right)^{\frac{n}{2}}. \label{PA1}
\end{equation} 
The series \eqref{EXP2} converges for $|w-h^+(0,t)|<h^+(0,t)-h^-(0,t)$. Since $|\lambda(t)-h^+(0,t)|\leq k(h^+(0,t)-h^-(0,t))$, $k<1$,
the series in the right hand side of \eqref{PA1} converges uniformly. So we can take  the limit under the summation symbol.
According to Corollary 2, $b_n(t)t^{-1/(b-c)^2+(n-1)/4+\varepsilon}\to0$ as
$t\to0$ for every $\varepsilon>0$. Therefore, given $\varepsilon>0$, we obtain
$\gamma(t)t^{-1/(b-c)^2-1/4+\varepsilon}\to0$ as $t\to0$, which completes the proof.
\end{proof}

Let us discuss now the posed question on different parametrizations of the slit $\gamma$. Namely, we
assume that $\gamma$ is an analytic curve together with its symmetric reflection and with the tip
at the origin. This means that the function $\gamma(s)$ is analytic in $[0,S]$ where $s$ is
the length parameter. Another function $\gamma(t)$ is analytic in $(0,T]$. We will study the
singularity type of $s=s(t)$ at $t=s=0$.

\begin{lemma}\label{lemma2} Let the L\"owner differential equation (\ref{Loewner4}) with the driving term
$\lambda\in\text{Lip}(1/2)$, generate slit maps $h(z,t): \mathbb H\setminus\gamma(t)\to\mathbb
H$, where $\gamma(t)\cup\overline{\gamma}(t)\cup\gamma(0)$ is an analytic curve which is mapped
onto $[h^-(0,t),h^+(0,t)]$. Suppose that $$\lim_{t\to+0}\frac{\lambda(t)}{\sqrt t}=c\geq 0,\;\;\;
\lim_{t\to+0}\frac{h^+(0,t)}{\sqrt t}=b,\;\;\;c<b\leq\frac{c+\sqrt{c^2+16}}{2}.$$ Then,
given $\varepsilon>0$, we have $$\lim_{t\to+0}s(t)t^{-\frac{1}{(b-c)^2}-\frac{1}{4}+\varepsilon}=0.$$
\end{lemma}

\begin{proof} The function $h^{-1}(w,t)$ is a one-to-one map of  the segment $[\lambda(t),h^+(0,t)]$ onto
$\gamma=\gamma(t)$. The length $s=s(t)$ of $\gamma(t)$ equals
$$s(t)=\int_{\lambda(t)}^{h^+(0,t)}|(h^{-1}(w,t))'|dw= \int_{\lambda(t)}^{h^+(0,t)}
\left|\left(\sum_{n=1}^{\infty}b_n(t)(w-h^+(0,t))^{n/2}\right)'\right|dw$$ $$
=\frac{1}{2}\int_{\lambda(t)}^{h^+(0,t)}
\left|\sum_{n=1}^{\infty}nb_n(t)(w-h^+(0,t))^{\frac{n}{2}-1}\right|dw$$ $$
\leq\frac{1}{2}\int_{\lambda(t)}^{h^+(0,t)}
\sum_{n=1}^{\infty}n|b_n(t)|(h^+(0,t)-w)^{\frac{n}{2}-1}dw=
\sum_{n=1}^{\infty}|b_n(t)|(h^+(0,t)-\lambda(t))^{\frac{n}{2}}.$$ This implies that, for every
$\varepsilon>0$, we have \begin{equation}s(t)t^{-\frac{1}{(b-c)^2}-\frac{1}{4}+\varepsilon}\leq
\sum_{n=1}^{\infty}|b_n(t)|t^{-\frac{1}{(b-c)^2}+\frac{n-1}{4}+\varepsilon}
\left(\frac{h^+(0,t)-\lambda(t)}{\sqrt t}\right)^{\frac{n}{2}}. \label{PA2}
\end{equation} Therefore, given $\varepsilon>0$, the limit $s(t)t^{-1/(b-c)^2-1/4+\varepsilon}\to0$ holds as
$t\to+0$, which
completes the proof. \end{proof}

\begin{theorem}\label{th3}
Let the L\"owner differential equation \eqref{Loewner4} generate slit maps $h(z,t): \mathbb H\setminus\gamma(t)\to \mathbb H$,
where $\gamma(t)\cup\bar{\gamma}(t)\cup\gamma(0)$ is an analytic curve which is mapped onto $[h^-(0,t),h^+(0,t)]$.
Then for the arc-length parameter $s$, $s(t)=A\sqrt{t}+o(\sqrt{t})$, $A\neq 0$, as $t\to +0$.
\end{theorem}
\begin{proof}
Let us consider the slit domain $B=\mathbb H\setminus \gamma_s$ in the $z$-plane, with $\gamma_s$ parametrized in the interval $[0,S]$ 
by the arc-length parameter $s$ as $\gamma(s)=x(s)+iy(s)$, where $x(s)$ and $y(s)$ are analytic in $[0,S]$ and $x(s)$ is even, $y(s)$ is odd. 
The slit $\gamma_s$ has another parametrization $\gamma(t)$, $t\in [0,T]$, according to the solution $h(z,t)$ to the corresponding
L\"owner equation \eqref{Loewner4}. Let us  turn to the slit domain $\tilde{B}$ in the $\zeta$-plane where $\zeta(z)=\sqrt{z^2-1/4}$. The domain
$\tilde{B}$ is given by eliminating from $\mathbb H$ the slit along the interval $[0,\frac{i}{2}]$, and the arc $\tilde{\gamma}$, which is the image of $\gamma$
under the map $\zeta(z)$.

Let the slit $[0,\frac{i}{2}]\cup\tilde{\gamma}$ be parametrized by a parameter $\tau$, so that
the function $\tilde{h}(\zeta,\tau): \tilde{B}\to\mathbb H$ with the hydrodynamic
normalization $\tilde{h}(\zeta,t)=\zeta+2\tau/\zeta+O(1/\zeta^2)$ near infinity solves the
corresponding L\"owner equation (4). According to the result of Earle and Epstein [2], the
arc-length parameter $\sigma=\sigma(\tau)$ of the slit is $C^1$-smooth for $\tau>0$. For
$0\leq \tau\leq 1/16$, we have $\tilde{h}(\zeta,\tau)=\sqrt{\zeta^2+4\tau}$, which gives
$\sigma=2\sqrt\tau$. The arc-length parameter $\sigma$ is connected with $s$ as
\[
s^2=\sigma-1/2+o(\sigma-1/2)
\]
near $\sigma=1/2$. From the other hand, for $\tau\geq1/16$,
\[
\tilde{h}(\zeta,\tau)=h(\sqrt{\zeta^2+1/4},t)=
\sqrt{\zeta^2+1/4}+\frac{2t}{\sqrt{\zeta^2+1/4}}+O\left(\frac{1}{\sqrt{\zeta^2+1/4}}\right)
\]
\[
=\zeta+\frac{1/8+2t}{\zeta}+O(1/\zeta)=\zeta+\frac{2\tau}{\zeta}+O(1/\zeta),
\]
which implies that $\tau=t+1/16$. Since the whole slit
$[0,\frac{i}{2}]\cup\tilde{\gamma}(\sigma)$ is $C^1$-smooth, it follows that the corresponding
driving term $\tilde{\lambda}(\tau)$ for $\tilde{h}(\zeta,\tau)$ in (4) is $C^1$ for $\tau>0$,
$\tilde{\gamma}(\tau)=\tilde{h}^{-1}(\tilde{\lambda}(\tau),\tau)\in C^1$, and $\sigma(\tau)\in
C^1$ for $\tau>0$. This implies that
\[\sigma-1/2=A_1\left(\tau-\frac{1}{16}\right)+o\left(\tau-\frac{1}{16}\right)=A_1t+o(t)=s^2+o(t)
\]
near $t=0$ which completes the proof.
\end{proof}

Comparing Theorem \ref{th3} with Lemma \ref{lemma2}, and observing that $\frac{1}{(b-c)^2}\geq 1/4$, where the equality sign is attained
only for $c=0$, $b=2$, we deduce that Lemmas \ref{lemma1} and \ref{lemma2} are valid only for  $c=0$, $b=2$. Therefore, we come  to the following theorem.

\begin{theorem}
Let the L\"owner differential equation \eqref{Loewner4} generate slit maps $h(z,t): \mathbb H\setminus\gamma(t)\to \mathbb H$,
where $\gamma(t)\cup\bar{\gamma}(t)\cup\gamma(0)$ is an analytic curve which is mapped onto $[h^-(0,t),h^+(0,t)]$.
Suppose that the limits 
$$\lim\limits_{t\to+0}\frac{\lambda(t)}{\sqrt t}=c\geq 0,\;\;\;
\lim\limits_{t\to+0}\frac{h^+(0,t)}{\sqrt t}=b$$ exist. Then $c=0$, $b=2$, $\lim\limits_{t\to+0}s(t)t^{-1/2}$
is finite, and given $\epsilon>0$, we have
\[
\lim\limits_{t\to+0}\gamma(t)t^{-\frac{1}{2}+\epsilon}=0.
\]
\end{theorem}

The latter theorem generalizes the results of \cite{Kadanoff} which are given for the particular cases of slits. One of them
is a rectilinear slit and the other one is a circular arc, both orthogonal to the real axis.

\end{document}